\documentclass[12pt,twoside]{article}
\usepackage{times}
\usepackage{amsmath,amssymb}
\usepackage{amsthm}
\usepackage{mathrsfs}
\usepackage{color}
\usepackage{exscale}
\usepackage{relsize}
\usepackage{color}
\usepackage{fancyhdr}
\allowdisplaybreaks

\newcommand{\ef}{ \hfill $ \Box $ \vskip 3mm}
\newcommand{\be}{\begin{equation}}
\newcommand{\ee}{\end{equation}}
\newcommand{\bea}{\begin{eqnarray}}
\newcommand{\eea}{\end{eqnarray}}

\newcommand{\bR}{{\mathbb R}}

\def\nn{\nonumber}

\def\ve{\varepsilon}
\def\la{\lambda}

\def\t{\tilde}
\def\q{\quad}
\def\qq{\qquad}
\def\th{\theta}

\def\G{\Gamma}
\def\dl{\delta}
\def\Dl{\Delta}
\def\ve{\varepsilon}
\def\fm{\mathfrak{m}}
\def\fx{\mathfrak{x}}
\def\ft{\mathfrak{t}}
\def\fz{\mathfrak{z}}
\def\fr{\mathfrak{r}}
\def\fy{\mathfrak{y}}
\def\fs{\mathfrak{s}}

\def\lt{\left}
\def\les{\lesssim}
\def\rt{\right}

\def\i{\infty}

\def \ls{\lesssim}
\def\p{\partial}
\def\f{\frac}
\def\na{\nabla}
\def\al{\alpha}

\def\O{\Omega}

\def\s{\sqrt}
 \allowdisplaybreaks

\begin{document}
 \footskip=0pt
 \footnotesep=2pt
\let\oldsection\section
\renewcommand\section{\setcounter{equation}{0}\oldsection}
\renewcommand\thesection{\arabic{section}}
\renewcommand\theequation{\thesection.\arabic{equation}}
\newtheorem{theorem}{\noindent Theorem}[section]
\newtheorem{lemma}{\noindent Lemma}[section]
\newtheorem*{claim}{\noindent Claim}
\newtheorem{proposition}{\noindent Proposition}[section]
\newtheorem*{inequality}{\noindent Weighted Inequality}
\newtheorem{definition}{\noindent Definition}[section]
\newtheorem{remark}{\noindent Remark}[section]
\newtheorem{corollary}{\noindent Corollary}[section]
\newtheorem{example}{\noindent Example}[section]

\title{Some Remarks on Regularity Criteria of Axially Symmetric Navier-Stokes Equations}

\author{Zijin Li $^{a,b,}$\footnote{E-mail:zijinli@smail.nju.edu.cn} ,\quad Xinghong Pan$^{c,}$\footnote{E-mail:xinghong{\_}87@nuaa.edu.cn}\vspace{0.5cm}\\
 \footnotesize $^a$Department of Mathematics and IMS, Nanjing University, Nanjing 210093, China.\\
 \footnotesize $^b$Department of Mathematics, University of California, Riverside, CA, 92521, USA.\\
\footnotesize $^c$Department of Mathematics, Nanjing University of Aeronautics and Astronautics, Nanjing 211106, China.
\vspace{0.5cm}
}

\date{}

\maketitle

\centerline {\bf Abstract}
Two main results will be presented in our paper. First, we will prove the regularity of solutions to axially symmetric Navier-Stokes equations under a $log$ supercritical assumption on the horizontally radial component $u^r$ and vertical component $u^z$, accompanied by a $log$ subcritical assumption on the horizontally angular component $u^\theta$ of the velocity. Second, the precise Green function for the operator $-(\Delta-\frac{1}{r^2})$ under the axially symmetric situation, where $r$ is the distance to the symmetric axis, and some weighted $L^p$ estimates of it will be given. This will serve as a tool for the study of axially symmetric Navier-Stokes equations. As an application, we will prove the regularity of solutions to axially symmetric Navier-Stokes equations under a critical (or a subcritical) assumption on the angular component $w^\theta$ of the vorticity.
\vskip 0.3 true cm

\vskip 0.3 true cm

\vskip 0.3 true cm


\section{Introduction}

The 3D incompressible Navier-Stokes equations are given by
\be\label{1.1}
\lt\{
\begin{aligned}
&\p_t u+u\cdot\na u+\na p-\Dl u=0,\\
&\na\cdot u=0,
\end{aligned}
\rt.
\ee
where $u(x,t)\in\bR^3,p(x,t)\in\bR$ represent the velocity vector and the scalar pressure respectively. The Navier-Stokes equations, which describe the motion of viscous fluid substances, are fundamental nonlinear partial differential equations in nature but are far from being fully understood. The global regularity problem of solutions for the 3D Navier-Stokes equations with smooth initial data remains open and is viewed as one of the most important open questions in mathematics \cite{Fcl1}.

The Navier-Stokes equations have the following scaling property: if $u(x,t),p(x,t)$ are solutions of \eqref{1.1}, then $u^\la(x,t)=\la u(\la x,\la^2 t), p^\la(x,t)=\la^2p(\la x,\la^2t)$ are also solutions. By multiplying both sides of $\eqref{1.1}_1$ with $u$ and integrating the resulted equation on $\bR^3$, we can see that smooth solutions, decaying fast enough at infinity, satisfy the following energy identity:
\be
\f{1}{2}\int_{\bR^3}|u(x,t)|^2dx+\int^t_0\int_{\bR^3}|\na u(x,s)|^2dxds=\f{1}{2}\int|u(x,0)|^2dx.\nn
\ee

This estimate seem to be the only useful a priori estimate for smooth solutions to the Navier-Stokes equations \eqref{1.1}. The main difficulty of proving the global regularity of solutions for the 3D Navier-Stokes equations lies in the fact that the above a prior estimate is supercritical with respect to the invariant scaling of the equations:
\be
\int_{\bR^3}|u^\la(x,t)|^2dx=\la^{-1}\int_{\bR^3}|u(x,t)|^2dx,\nn
\ee
\be
 \int^\i_0\int_{\bR^3}|\na u^{\la}(x,s)|^2dxds=\la^{-1}\int^\i_0\int_{\bR^3}|\na u(x,s)|^2dxds.\nn
\ee

In the cylindrical coordinates $(r,\th, z)$, we have $x=(x_1,x_2,x_3)=(r\cos\th,r\sin\th,z)$ and the axi-symmetric solution of the incompressible Navier-Stokes equations is given as
\[
u=u^r(r,z,t)e_r+u^{\th}(r,z,t)e_{\th}+u^z(r,z,t)e_z,
\]
where the basis vectors $e_r,e_\th,e_z$ are
\[
e_r=(\frac{x_1}{r},\frac{x_2}{r},0),\quad e_\th=(-\frac{x_2}{r},\frac{x_1}{r},0),\quad e_z=(0,0,1).
\]
The components $u^r,u^\th,u^z$ satisfy
\begin{equation}\label{1.2}
\left\{
\begin{aligned}
&\p_t u^r+(b\cdot\nabla)u^r -\frac{(u^\th)^2}{r}+\p_r p=(\Delta-\frac{1}{r^2})u^r, \\
&\p_t u^\th+(b\cdot\nabla) u^\th+\frac{u^\th u^r}{r}=(\Delta-\frac{1}{r^2})u^\th , \\
&\p_t u^z+(b\cdot\nabla)u^z+\p_z p=\Delta u^z ,                                    \\
&b=u^re_r+u^ze_z,\q \nabla\cdot b=\p_ru^r+\frac{u^r}{r}+\p_zu^z=0.
\end{aligned}
\right.
\end{equation}

We can also compute the axi-symmetric vorticity $w=\nabla\times u=w^re_r+w^\th e_\th+w^ze_z$  as follows
\[
w^r=-\p_z u^\th, \ w^\th=\p_z u^r-\p_r u^z,\  w^z=(\p_r+\frac{1}{r})u^\th.
\]
The equations for $w^r, w^\th, w^z$ are
\be\label{e1.3}
\lt\{
\begin{aligned}
&\p_t w^r+(b\cdot\na)w^r-(\Dl-\f{1}{r^2})w^r-(w^r\p_r+w^z\p_z)u^r=0,\\
&\p_t w^\th+(b\cdot\na)w^\th-(\Dl-\f{1}{r^2})w^\th-\f{u^r}{r}w^\th-\f{1}{r}\p_z(u^\th)^2=0,\\
&\p_t w^z+(b\cdot\na)w^z-\Dl w^z-(w^r\p_r+w^z\p_z)u^z=0.
\end{aligned}
\rt.
\ee

Our paper's first aim is to study the regularity of axially symmetric Navier-Stokes equations under a supercritical assumption on the drift term $b$ and a subcritical assumption on the angular component $u^\th$, namely:
\begin{equation}\label{e1.2}
|b|\lesssim\frac{(1+|\ln r|)^{\beta}}{r}, \q |u^\th|\lesssim \frac{(1+|\ln r|)^{-\al}}{r}
\end{equation}
where $0\leq\beta<\al/6$ and without loss of generality, $\al\in(0,1]$ is a small constant. Here is the theorem:
\begin{theorem}\label{Thm1}
Let $(u,p)$ be a suitable weak solution of the axisymmetric Navier-Stokes equation \eqref{1.2} in $\bR^3\times [-1,0]$. Assume that $u$ satisfies \eqref{e1.2}. Then we have
\[
\sup\limits_{(x,t)\in\bR^3\times[-1,0)}|u|< +\i.
\]
\end{theorem}

\qed

\noindent Readers can refer to \cite{CSTY1} for the definition of suitable weak solutions.

Besides, from  the Biot-Savart law, we have that $-\Dl u=\na\times w$. under the axially symmetric situation, we can get
\be\label{1.4}
-(\Dl-\f{1}{r^2})u^r=-\p_z w^\th,\ -(\Dl-\f{1}{r^2})u^\th=\p_z w^r-\p_r w^z,\ -\Dl u^z=\f{1}{r}\p_r(rw^\th).
\ee

We see that the operator $-(\Dl-\f{1}{r^2})$ plays an important part in the relationship between $u$ and $w$, which also appears in \eqref{1.2} and \eqref{e1.3}. So a precise Green function of the operator is necessary for us to study the axially Navier-Stokes equations.

Our second target is to calculate the precise formula of the Green function of the following elliptic operator with a inverse-square potential:
\be
\mathcal{L}:=-\left(\Delta-\frac{1}{r^2}\right) \nn
\ee
and give some weighted $L^p$ estimates of it. Here is the result:
\begin{theorem}\label{Thm2}
The Green function of the operator $\mathcal{L}:=-\left(\Delta-\frac{1}{r^2}\right)$ has the following representation formula
\be\label{GreenFunc}
\Gamma(r,\rho,z-l)=\int^\i_0G(t;r,\rho,z-l) dt,
\ee
where
\be\label{HeatKer}
G(t;r,\rho,z-l)=\frac{1}{4\sqrt{\pi}t^{3/2}}\exp\left(-\frac{r^2+\rho^2+(z-l)^2}{4t}\right)\mathcal{I}_1\left(\frac{r\rho}{2t}\rt)
\ee
is the heat kernel of the operator $\p_t-(\Dl-\f{1}{r^2})$ and $\mathcal{I}_\alpha$ is the \emph{modified Bessel function of first kind} with footnote $\alpha\in\mathbb{R}$.

Besides, we have the following weighted $L^p$ estimates for $\G$.
\be\label{1.8}
\left(\int_{-\infty}^\infty\int_0^\infty|\G(r,\rho,z-l)|^p\f{1}{\rho}d\rho  dl\right)^{1/p}\leq Cr^{1/p-1},\q \text{for}\q 1\leq p< 2,
\ee
\be\label{1.10}
\left(\int_{-\infty}^\infty\int_0^\infty|\G(r,\rho,z-l)|^2\rho d\rho  dl\right)^{1/2}\leq C\s{r},
\ee
and
\be\label{1.9}
\int_{-\infty}^\infty\int_0^\infty|\p_z\G(r,\rho,z-l)|\f{1}{\rho^{\dl}}d\rho  dl\leq Cr^{-\dl},\q \dl\in[0,1),
\ee
\end{theorem}

\qed

\begin{remark}
It seems that the 2-dimensional version of the heat kernel \eqref{HeatKer} was firstly calculated by T. Gallay and V. Sverak in \cite{GS}, by using a different approach.
\end{remark}

\qed

\begin{remark}
Readers can see \cite{AS1} for the definition of the modified Bessel function. $\mathcal{I}_1(s)$ have the following formula and asymptotic behavior.
\be\label{1.12}
\mathcal{I}_1(s)=\sum_{m=0}^\infty\frac{1}{m!(m+1)!}\left(\frac{s}{2}\right)^{2m+1}.
\ee
and
\be\label{1.13}
\mathcal{I}_1\left(s\right)\lesssim
\left\{
\begin{array}{ll}
s&,\quad 0<s\leq 1;\\
\frac{e^s}{\sqrt{s}}&,\quad s>1.\\
\end{array}
\right.
\ee
\end{remark}

\qed

As a corollary of \eqref{1.8} and \eqref{1.9}, we have the $\theta-$direction of stream function is bounded under the condition $|\omega^\theta(t,r,z)|\lesssim r^{-2}$ and $|u^r(t;r,z)|\ls \f{1}{r^\dl}$ under the condition $|\omega^\theta(t,r,z)|\lesssim \f{1}{r^{1+\dl}}$ for $\dl\in[0,1)$.
\begin{corollary}\label{Cor1}
Suppose $\omega^\theta=\omega^\theta(t,r,z)$ satisfies $|\omega^\theta(t,r,z)|\lesssim r^{-2}$, then the stream function $L^\theta$ of the velocity field $b=u^r e_r+u^ze_z$, defined by $b:=\na\times (L^\theta e_\th)$, satisfies
\be\label{1.14}
|L^\th|\leq C
\ee
Moreover, the solution $u$ of axially symmetric Navier-Stokes equations is regular.
\end{corollary}
\qed
\begin{corollary}\label{Cor2}
Suppose $\omega^\theta=\omega^\theta(t,r,z)$ satisfies $|\omega^\theta(t,r,z)|\lesssim \f{1}{r^{1+\dl}}$ for $\dl\in[0,1)$, then the horizontally radial component $u^r$ of the velocity satisfies
\be\label{1.15}
|u^r|\ls \f{1}{r^\dl}.
\ee
Moreover, the solution $u$ of axially symmetric Navier-Stokes equations is regular.
\end{corollary}
\qed

\begin{remark}
To the best of our knowledge, the result in \textbf{Corollary \ref{Cor1}} was firstly realized by Zhen Lei and Qi S. Zhang \cite{LZ4}, by using a heat kernel estimate derived by Alexander Grigor'yan \cite{Ga1}. We proved here in an alternative way.
\end{remark}
\qed

Before ending our introduction, we recall some regularity results on the axisymmetric Navier-Stokes equations. Under the no swirl assumption, $u^\th=0$ , Ladyzhenskaya \cite{Loa1} and Ukhovskii-Iudovich \cite{UI1} independently proved that weak solutions are regular for all time. When the swirl $u^\th$ is non-trivial, some efforts and progress have been made on the regularity of the axisymmetric solutions. In \cite{CSTY1,CSTY2}, Chen-Strain-Yau-Tsai proved that the suitable weak solutions are regular if the solution satisfies $r|u|\ls 1$. Their method is based on the ones of De Giorgi, Nash and Moser. Also, Koch-Nadirashvili-Seregin-Sverak in \cite{KNSS1} proved the same result, by using a Liouville theorem and scaling-invariant property. Lei-Zhang in \cite{LZ2} proved regularity of the solution under a more general assumption on the drift term $b$ where $b\in L^\i\left([-1,0),BMO^{-1}\right)$. Seregin-Zhou \cite{SZ1} prove that any axially symmetric suitable weak solution $u$, belonging to $L^\i(0,T;\dot{B}^{-1}_{\i,\i})$, is
smooth. Pan \cite{Px1} proved the regularity of solutions under a slightly supercritical assumption on the drift term $b$. Recently, Chen-Fang-Zhang in \cite{CFZ} proved that if $ru^\th$ satisfies $r|u^\th|\leq Cr^\al, \al>0$, then $u$ is regular without any other a prior assumptions. As a complementary of their work, Pan \cite{Px2} proved the regularity of solutions by assuming $r|u^r|\leq Cr^\al$ or $r|u^z|\leq Cr^\al, \al>0$. Later, Lei-Zhang in \cite{LZ1} improved the result in \cite{CFZ} by assuming $r|v^\th|\leq C|\ln r|^{-2}$ for small $r$. Also Wei in \cite{Wei} improved the $log$ power from $-2$ to $-\f{3}{2}$.

When the initial data satisfies some integral conditions, Abidi-Zhang in \cite{AZ1} give the global smooth axially symmetric solutions of 3-D inhomogeneous incompressible Navier-Stokes equations. From the partial regularity theory of \cite{Caffarelli01}, any singular points of the axis-symmetric suitable weak solution can only lie on the symmetric axis. In \cite{Buiker01}, Burke-Zhang give a priori bounds for the vorticity of axially symmetric solutions which indicates that the result of \cite{Caffarelli01} can be applied to a large class of weak solutions. Neustupa and Pokorny \cite{Neustupa01} proved certain regularity of one component (either $u^\th$ or $u^r$) imply regularity of the other components of the solutions. Chae-Lee \cite{CL} proved regularity assuming a zero-dimensional integral norm on $w^\th$: $w^\th\in L^s_tL^q_x$ with $3/q +2/s=2$. Also regularity results come from the work of Jiu-Xin \cite{Jiu01} under the assumption that another zero-dimensional scaled norms $\int_{Q_R}(R^{-1}|w^\th|^2+R^{-3}|u^\th|^2) dz$ is sufficiently small for $R>0$ is small enough. On the other hand, Lei-Zhang \cite{LZ3} give a structure of singularity of 3D axis-symmetric equation near a maximum point. Tian-Xin \cite{Tian01} constructed a family of singular axi-symmetric solutions with singular initial data. Hou-Li \cite{Hou02} construct a special class of global smooth solutions. See also a recent extension: Hou-Lei-Li \cite{Hou01}.

Our paper is organized as follows. In Section 2, we give the proof of \textbf{Theorem \ref{Thm1}}, and Section 3 is devoted to proving \textbf{Theorem \ref{Thm2}}, \textbf{Corollary \ref{Cor1}} and \textbf{Corollary \ref{Cor2}}. Throughout the paper, we use $C$ to denote a generic constant which may be different from line to line. We also apply $A\lesssim B$ to denote $A\leq CB$.

\section{Proof of Theorem 1.1}\label{SEC2}

\q In this section we will prove \textbf{Theorem \ref{Thm1}} and get the regularity of the solution under the assumption \eqref{e1.2}. The idea comes from [Chen-Strain-Tsai-Yau]'s proof in \cite{CSTY2} where they assume $|u|\leq Cr^{-1}$.

We divide the proof into 3 steps.\\
{\bf Step one: scaling of the solution and set up of an equation}

Let $\mathfrak{m}$ be the maxmium of $|u|$ up to a fixed time $t_0$ and we may assume $\fm>1$ is large. Define the scaled solution
\be
u_\fm(\mathfrak{x},\mathfrak{t})=\fm^{-1}u(\frac{\mathfrak{x}}{\fm},\frac{\mathfrak{t}}{\fm^2}),\q \fx=(\fx_1,\fx_2,\fz).  \nn
\ee
Denote $x=(x_1,x_2,z)$ and $\fx=(\fx_1,\fx_2,\fz)$, $r=\s{x^2_1+x^2_2}$ and $\fr=\s{\fx^2_1+\fx^2_1}$. We have the following estimate for $r$ and $\fr$ for time $t<t_0$ and $\ft< \fm^2t_0$:
\be
|\nabla^k u_\fm|\leq C_k.   \label{e4.1}
\ee
This inequality follows from $\|u_\fm\|_{L^\infty}\leq1$ for $t<t_0$ and the standard regularity theorem of Navier-Stokes equations. Its angular component (we omit
the time dependence below) $u^\th_\fm(\fr,\fz)$ satisfies $u^\th_\fm(0,\fz)=0=\p_\fz u^\th_\fm(0,\fz)$ for all $\fz$. By mean value theorem and \eqref{e4.1},
\be
|u^\th_\fm(\fr,\fz)|\lesssim \fr,\ |\p_\fz u^\th_\fm(\fr,\fz)|\lesssim \fr \q \rm{for} \ \fr\leq1. \nn
\ee
Together with \eqref{e4.1} for $\fr\geq1$, we get
\be
|u^\th_\fm|\les \f{\fr}{1+\fr} , \  |\p_\fz u^\th_\fm|\les \f{\fr}{1+\fr}.  \label{e4.2}
\ee
Then $u^\th_\fm(\fr,\fz)$ satisfies the estimate
\be
|u^\th_\fm(\fr,\fz)|=\fm^{-1}\Big|u^\th(\frac{\fx}{\fm},\frac{\ft}{\fm^2})\Big|\lesssim
\frac{(1+|\ln\frac{\fr}{\fm}|)^{-\al}}{\fr}\nn
\ee
Combining this with \eqref{e4.2}, one has
\be
|u^\th_\fm(\fr,\fz)|\les \min\Big\{\f{\fr}{1+\fr},\frac{(1+|\ln\frac{\fr}{\fm}|)^{-\al}}{\fr}\Big\}\nn
\ee
Let $\fr_0$ be such that $\f{\fr_0}{1+\fr_0}=\frac{(1+|\ln\frac{\fr_0}{\fm}|)^{-\al}}{\fr_0}$. It is not hard to see that there exists a constant $C>1$ such that
\be\label{4.4}
C^{-1}(\ln \fm)^{-\al/2}\leq \fr_0\leq C(\ln \fm)^{-\al/2}.
\ee
Then we can rewrite the estimate of $u^\th_\fm$ as follows
\be\label{e4.4}
u^\th_\fm(\fr,\fz)\lesssim\lt\{
\begin{aligned}
&\f{\fr}{1+\fr}\q\q\q\qq 0<\fr\leq \fr_0,\\
&\frac{(1+|\ln\f{\fr}{\fm}|)^{-\al}}{\fr} \q \fr\geq \fr_0.
\end{aligned}\rt.
\ee
Now consider the angular component of the rescaled vorticity. Recall $\Omega=\frac{w_\th}{r}$. Let
\bea
\Omega_\fm(\fx,\ft)=\frac{w^\th_\fm(\fx,\ft)}{\fr}=\fm^{-2}w^\th(\frac{\fx}{\fm},\frac{\ft}{\fm^2})\f{1}{\fr}.  \nn
\eea
Note that $w^\th_\fm$ and $\nabla w^\th_\fm$ are bounded by \eqref{e4.1} and also $w^\th_\fm|_{\fr=0}=0$, so one has
\be
|\O_\fm|\lesssim\frac{1}{1+\fr}.  \nn
\ee
$\O_\fm$ satisfies
\be
(\p_\ft-L)\O_\fm=f,\q L=\Delta+\frac{2}{\fr}\p_\fr-b_\fm\cdot\nabla,     \nn
\ee
where $f=\fr^{-2}\p_\fz(u^\th_\fm)^2$ and $b_\fm=u^\fr_\fm e_\fr+u^\fz_\fm e_\fz$, $|b_m|\leq 1$.\\
Combining the estimates \eqref{e4.2} and \eqref{e4.4}, one has
\be\label{e4.5}
f=\frac{2}{\fr^2}u^\th_\fm\p_\fz u^\th_\fm\les\lt\{
\begin{aligned}
&1\q\qq\qq\q\ \ 0<\fr\leq \fr_0,\\
&\frac{(1+|\ln\f{\fr}{\fm}|)^{-\al}}{(1+\fr)\fr^2} \q \fr\geq \fr_0.
\end{aligned}\rt.
\ee
Let $P(\fx,\ft;\fy,\fs)$ be the kernel of $\p_\ft-L$. By Duhamel's formula
\bea\label{e4.60}
\O_\fm(\fx,\ft)&=&\int P(\fx,\ft;\fy,\fs)\O_\fm(\fy,\fs)d\fy+\int^\ft_\fs\int P(\fx,\ft;\fy,\tau)f(\fy,\tau)d\fy d\tau  \nn  \\
      &:=&I_1+I_2.
\eea
{\bf Step two: bounding of $\O_\fm$ }
\\
In the following ,we will estimate \eqref{e4.60} and give a bound for $\O_\fm(\fx,\ft)$.

The kernel $P(\fx,\ft;\fy,\fs)$ satisfies $P\geq0,\ \int P(\fx,\ft;\fy,\fs) d\fy\leq1$ and
\be
P(\fx,\ft;\fy,\fs)\leq C(\ft-\fs)^{-3/2}\exp\left\{-C\frac{|\fx-\fy|^2}{\ft-\fs}\left(1-\frac{\ft-\fs}{|\fx-\fy|}\right)^2_+\right\}.\label{e4.6}
\ee

The proof of estimate \eqref{e4.6} is based on \cite{CL1}, but due to the singularity of the term $\f{2}{r}\p_r$, the proof is more involved. See Theorem 3 in \cite{Px1}.

Now we give estimates of $P$ in two cases.

From \eqref{e4.6}, it is easy to see that
\be\label{2.8}
P(\fx,\ft;\fy,\fs)\les (\ft-\fs)^{-3/2}\lt\{
\begin{aligned}
&\exp\big\{-c\frac{|\fx-\fy|^2}{\ft-\fs}\big\},\q |\fx-\fy|>2(\ft-\fs);\\
& 1,\qq\qq\qq\qq\ \ |\fx-\fy|\leq 2(\ft-\fs).
\end{aligned}
\rt.
\ee
With the estimate \eqref{2.8} and H\"{o}lder inequality, one gets, when $\ft-\fs>1$,
\bea\label{e4.7}
|I_1|&\leq&\left[\int P(\fx,\ft;\fy,\fs)|\O_\fm(\fy,\fs)|^{2+\dl}d\fy\right]^{\frac{1}{2+\dl}}  \left[\int P(\fx,\ft;\fy,\tau)d\fy\right]^{\frac{1+\dl}{2+\dl}}\nn \\
   &\les&\left[\Big(\int_{|\fx-\fy|>2(\ft-\fs)|}+\int_{|\fx-\fy|<2(\ft-\fs)}\Big) P(\fx,\ft;\fy,\fs)|\O_\fm(\fy,\fs)|^{2+\dl}d\fy\right]^{\frac{1}{2+\dl}}\nn\\
   &\les&(\ft-\fs)^{-\frac{3}{2(2+\dl)}}\left\{\int_{|\fx-\fy|>2(\ft-\fs)}e^{-c\frac{|\fx_3-\fy_3|^2}{\ft-\fs}}\frac{\fr}{(\fr+1)^{2+\dl}}d\fr d\fy_3\rt.\nn\\
   &&\lt.\qq\qq\qq\qq +\int_{|\fx-\fy|<2(\ft-\fs)}\frac{\fr}{(\fr+1)^{2+\dl}}d\fr d\fy_3\right\}^{\f{1}{2+\dl}} \nn  \\
   &\lesssim&(\ft-\fs)^{-\frac{3}{2(2+\dl)}}\left\{(\ft-\fs)^{\frac{1}{2}}+(\ft-\fs)\right\}^{1/(2+\dl)}    \nn \\
   &\lesssim&(\ft-\fs)^{-\frac{1}{2(2+\dl)}}.
\eea
Next
\bea\label{4.90}
|I_2|&\leq&\int^\ft_\fs(\ft-\tau)^{-\frac{3}{2}}\left\{\int_{|\fx-\fy|\leq2(\ft-\tau)}|f|d\fy+\int_{|\fx-\fy|\geq2(\ft-\tau)}e^{-c\frac{|\fx-\fy|^2}{\ft-\tau}}|f|d\fy\right\}d\tau  \nn  \\
    &:=&I_{2,1}+I_{2,2}.
\eea
We deal with $I_{2,1},I_{2,2}$ in \eqref{4.90} as follows,
\bea\label{2.11e}
I_{2,1}&=&\int^\ft_\fs(\ft-\tau)^{-\frac{3}{2}}\int_{|\fx-\fy|\leq2(\ft-\tau)}|f|d\fy d\tau\nn\\
&\lesssim&\int^\ft_\fs(\ft-\tau)^{-\frac{3}{2}}\int^{+\i}_0\sup_{\fy_3}|f|\fr d\fr\int_{|\fx_3-\fy_3|\leq 2(\ft-\tau)}d \fy_3d\tau \nn \\
         &\lesssim&(\ft-\fs)^{1/2}\int^{+\infty}_0\sup_{\tau,\fy_3}|f|\fr d\fr   \nn  \\
         &\lesssim&(\ft-\fs)^{1/2}\left\{\int^{\fr_0}_0 \fr d\fr+\int^\i_{\fr_0}\f{(1+|\ln\f{\fr}{\fm}|)^{-\al}}{(1+\fr)\fr}d\fr \right\}  \nn \\
         &\les&(\ft-\fs)^{1/2}\lt\{\fr^2_0+\int^1_{\fr_0}\f{(1+\ln\f{\fm}{\fr})^{-\al}}{\fr}d\fr\rt.\nn\\
         &&\lt.\qq\q\ +\int^{\fm}_{1}\f{(1+\ln\f{\fm}{\fr})^{-\al}}{\fr^2}d\fr+\int^\i_{\fm}\f{(1+\ln\f{\fr}{\fm})^{-\al}}{\fr^2}d\fr\rt\}\nn\\
\eea
Now we make a brief estimate for the integral on the right hand of \eqref{2.11e}. Since $\fr_0\approx (\ln \fm)^{-\al/2}$, we have
\be
\fr^2_0\ls (\ln \fm)^{-\al}. \nn
\ee
Also
\be
\int^1_{\fr_0}\f{(1+\ln\f{\fm}{\fr})^{-\al}}{\fr}d\fr\ls (1+\ln \fm)^{-\al} \int^1_{\fr_0}\f{1}{\fr}d\fr\ls (\ln \fm)^{-\al+\dl_1}.\nn
\ee
Due the decreasing of $\f{(1+\ln\f{\fm}{\fr})^{-\al}}{\s{\fr}}$ when $\fr\in[1,\fm]$, we have
\be
\int^{\fm}_{1}\f{(1+\ln\f{\fm}{\fr})^{-\al}}{\fr^2}d\fr\ls (1+\ln \fm)^{-\al}\int^{\fm}_{1}\fr^{-3/2}d\fr\ls  (\ln \fm)^{-\al}.\\
\ee
At last
\be
\int^\i_{\fm}\f{(1+\ln\f{\fr}{\fm})^{-\al}}{\fr^2}d\fr\ls \int^\i_{\fm}\f{1}{\fr^2}d\fr\ls \fm^{-1}.\nn
\ee
The above inequalities indicate that for $\dl_1>0$, which is sufficient small and independent on $\fm$, we get
\be\label{4.10}
I_{2,1}\les (\ft-\fs)^{1/2}(\ln \fm)^{-\al+\dl_1}.
\ee
\bea\label{2.12}
I_{2,2}&=&\int^\ft_\fs(\ft-\tau)^{-\frac{3}{2}}\int_{|\fx-\fy|\geq2(\ft-\tau)}e^{-c\f{|\fx-\fy|^2}{\ft-\tau}}|f|d\fy d\tau\nn\\
         &\lesssim&\int^\ft_\fs(\ft-\tau)^{-\frac{3}{2}}\lt(\int e^{-c\f{|\fx-\fy|^2}{\ft-\tau}}d\fy\rt)^{\f{\dl}{1+\dl}}\\
         &&\hskip 3cm\cdot\lt( \int_{|\fx-\fy|\geq2(\ft-\tau)}e^{-c\f{|\fx-\fy|^2}{\ft-\tau}}|f|^{1+\dl}d\fy \rt)^{\f{1}{1+\dl}}d\tau \nn \\
         &\lesssim&\int^\ft_\fs(\ft-\tau)^{-\frac{3}{2}}(\ft-\tau)^{\frac{3}{2}\f{\dl}{1+\dl}}(\ft-\tau)^{\frac{1}{2}\f{1}{1+\dl}}d\tau\lt(\int \sup_{\tau,\fy_3}|f|^{1+\dl}\fr d\fr\rt)^{\f{1}{1+\dl}}  \nn  \\
         &\lesssim&(\ft-\fs)^{\f{\dl}{1+\dl}}\lt(\int \sup_{\tau,\fy_3}|f|^{1+\dl}\fr d\fr\rt)^{\f{1}{1+\dl}}\nn\\
         &\les&(\ft-\fs)^{\f{\dl}{1+\dl}}(\ln\fm)^{-\al+\dl_2},
\eea
where the computation of the integral $\lt(\int \sup_{\tau,\fy_3}|f|^{1+\dl}\fr d\fr\rt)^{\f{1}{1+\dl}}$ is the same as that of $\int\sup_{\tau,\fy_3}|f|\fr$ $d\fr$ in \eqref{2.11e} and $\dl_2>0$ can be sufficient small and is independent on $\fm$.

Now, let $\fs=\ft-(\ln\fm)^{\f{4}{3}\al}>-\fm^2$(hence $\O_\fm$ is defined), From \eqref{e4.7}, \eqref{4.90}, \eqref{4.10}, \eqref{2.12} and by choosing sufficiently small $\dl,\dl_1,\dl_2>0$, we can get that for any small $\ve_1$,
\be
|\O_\fm(\fx,\ft)| \lesssim (\ln\fm)^{-\al/3+\ve_1} .\nn
\ee
{\bf Step three: bounding the solution $u$ from $\O_\fm$}\\
First
\be\label{2.13}
|w^\th(x,t)|= \fm^2|w^\th_\fm(r\fm,z\fm,t\fm^2)|=|\O_\fm(r\fm,z\fm,t\fm^2)|\fm^2r\fm\les \fm^3r(\ln \fm)^{-\al/3+\ve_1}.
\ee
In the following, we bound $b=u^re_r+u^ze_z$.\\
Denote $B_\rho(x_0)=\{x:|x-x_0|<\rho\}$, where $\rho>0$ to be determined later. By Biot-Savart law, $b$ satisfies
\be
-\Delta b={\rm curl}(w_\th e_\th).   \nn
\ee
From the estimates of elliptic equation \cite{GT1}, for $q>1$,
\be \label{e4.12}
\sup\limits_{B_\rho(x_0)}|b|\leq C\left(\rho^{-\frac{3}{q}}\|b\|_{L^q(B_{2\rho}(x_0))}+\rho\sup\limits_{B_{2\rho}(x_0)}|w_\th|\right).
\ee
For a fixed $\rho\ll 1$, to be determined later, set $x_0\in\{(r,\th,z):r<\rho\}$ and $1<q<2$.
By the assumption \eqref{e1.2} on $b$,
\bea\label{4.20}
\rho^{-\frac{3}{q}}\|b\|_{L^q(B_{2\rho}(x_0))}&\leq&\rho^{-\frac{3}{q}}\Big \|\frac{(1+|\ln r|)^{\beta}}{r}\Big\|_{L^q(B_{2\rho}(x_0))} \nn  \\
&\les& \rho^{-\frac{3}{q}}\left[\int^{z_0+2\rho}_{z_0-2\rho}dz\int^{3\rho}_0\frac{|\ln r|^{q\beta}}{r^q}rdr\right]^{\frac{1}{q}} \nn  \\
&\les& \rho^{-\frac{2}{q}}\left[\int^{3\rho}_0\frac{|\ln r|^{q\beta}}{r^{q-1}}dr\right]^{\frac{1}{q}}.
\eea
We compute $\int^{3\rho}_0\frac{|\ln r|^{q\beta}}{r^{q-1}}dr$ as follows,
\bea\label{2.16}
\int^{3\rho}_0\frac{|\ln r|^{q\beta}}{r^{q-1}}dr&=&\int^{+\infty}_{\frac{1}{3\rho}}(\ln r)^{q\beta}r^{q-3}dr \nn \\
&=&\left(\int^{\rho^{-2}}_{\f{1}{3}\rho^{-1}}+\int^{+\infty}_{\rho^{-2}}\right)(\ln r)^{q\beta}r^{q-3}dr       \nn  \\
&\les& (\ln\f{1}{\rho})^{q\beta}\int^{\rho^{-2}}_{\rho^{-1}}r^{q-3}dr+\int^{+\infty}_{\rho^{-2}}r^{q-3+\dl}dr \nn \\
&\les& \rho^{2-q}(\ln\f{1}{\rho})^{q\beta},
\eea
where $\dl>0$ is chosen to be sufficiently small such that $q-3+\dl<-1$. Inserting \eqref{2.16} into \eqref{4.20}, we can get
\be\label{2.17}
\rho^{-\frac{3}{q}}\|b\|_{L^q(B_{2\rho}(x_0))}\les \rho^{-1}\big(\ln\f{1}{\rho}\big)^{\beta}.
\ee
Now we set $\rho=\fm^{-1}(\ln\fm)^{\beta+\ve_2}$, where $0<\ve_2<\f{\al}{2}$. Then \eqref{2.17} implies that
\be\label{2.18}
\rho^{-\frac{3}{q}}\|b\|_{L^q(B_{2\rho}(x_0))}\les\fm(\ln\fm)^{-\ve_2}.
\ee
Also, by \eqref{2.13} and our choice of $\rho$, we have
\be\label{2.19}
\rho\sup\limits_{B_{2\rho}(x_0)}|w_\th|\les \fm(\ln\fm)^{2\beta-\f{\al}{3}+\ve_1+2\ve_2}
\ee
combining \eqref{e4.12}, \eqref{2.18} and \eqref{2.19}, we can get that for a sufficiently small $\ve>0$,\\
\textbf{when $\boldsymbol {r\leq \fm^{-1}(\ln \fm)^{\beta+\ve}}$,}
\be
\boldsymbol{|b(t,r,z)|\les \fm(\ln \fm)^{-\ve}.}\nn
\ee
From our assumption on $b$, we have \\
\textbf{when $\boldsymbol {r\geq \fm^{-1}(\ln \fm)^{\beta+\ve}}$,}
\be
\boldsymbol{|b(t,r,z)|\les \f{(1+|\ln r|)^\beta}{r}\leq \fm(\ln \fm)^{-\ve}.}\nn
\ee
The above two inequalities implies that
\be\label{2.20}
\boldsymbol{|b(t,r,z)|\leq \fm(\ln \fm)^{-\ve}.}
\ee
In the following, we bound $u_\th$. Recall the relationship between $u^\th(x,t)$ and $u^\th_\fm(\fx,\ft)$ and the estimate of $u^\th_\fm$ in \eqref{e4.4}, then we have
\bea
u_\th(r,z)&=&\fm|u^\th_\fm(r\fm,z\fm)|\nn\\
&\les& \fm\left\{
\begin{aligned}
&\f{r\fm}{1+r\fm},\qq\qq r<\frac{\fr_0}{\fm};  \\
&\f{(1+|\ln r|)^{-\al}}{(1+r\fm)r\fm},\qq r\geq\frac{\fr_0}{\fm}.
\end{aligned}
\right.\nn\\
&\les&\lt\{
\begin{aligned}
&\fm\fr_0,   \qq\qq r<\frac{\fr_0}{\fm};\\
&\f{(1+|\ln r|)^{-\al}}{r},\qq \frac{\fr_0}{\fm}\leq r\leq \f{1}{\fm};\\
&\f{(1+|\ln r|)^{-\al}}{r^2\fm},\qq\qq \f{1}{\fm}\leq r\leq 1;\\
&\fm^{-1},\qq\qq\qq r\geq 1.
\end{aligned}
\rt.\nn\\
&\les&\lt\{
\begin{aligned}
&\fm(\ln\fm)^{-\al/2},   \qq\qq r<\frac{\fr_0}{\fm};\\
&\fm(\ln\fm)^{-\al/2},\qq \frac{\fr_0}{\fm}\leq r\leq \f{1}{\fm};\\
&\fm(\ln\fm)^{-\al},\qq\qq \f{1}{\fm}\leq r\leq 1;\\
&\fm^{-1},\qq\qq\qq r\geq 1,
\end{aligned}
\rt.\nn
\eea
which indicates that
\be\label{2.21}
\boldsymbol{|u^\th(r,z)|\les \fm(\ln\fm)^{-\al/2}}.
\ee

Since $\fm$ is the maximum of $|u|$, $\fm=\max\{\sup |b|, \sup|u^\th|\}$. the estimates \eqref{2.20} and \eqref{2.21} together indicate that
\be
\fm\leq C\fm(\ln\fm)^{-\ve}.  \nn
\ee
This gives an upper bound for $\fm$ which completes the proof of \textbf{Theorem \ref{Thm1}}.  \ef

\section{Proof of Theorem 1.2,  Corollary 1.1 and Corollary 1.2}\label{SEC3}
\subsection{Calculation of the Green Function and the Heat Kernel}
This subsection is devoted to deducing the precise formula \eqref{HeatKer} in \textbf{Theorem \ref{Thm2}}. Consider the parabolic equation with a inverse-square potential:
\be\label{3.1}
\Delta v-\frac{1}{r^2}v-v_t=0.
\ee
Here $v=v(t,r,z)$ is axially symmetric. $\Delta=\p^2_r+\frac{1}{r}\p_r+\p^2_z$. Denoting $v=r\cdot f$, it follows that $f=f(t,r,z)$ satisfies
\be\label{3.2}
\Delta f+\frac{3}{r}\p_r f-f_t=0.
\ee
Therefore, if we denote $\Delta_5:=\p^2_{x_1}+\p^2_{x_2}+\p^2_{x_3}+\p^2_{x_4}+\p^2_{z}$ the 5-dimensional Laplacian, and $r=\sqrt{x^2_1+x^2_x+x^2_3+x^2_4}$ the distance between $x':=(x_1,x_2,x_3,x_4)$ and origin in $\mathbb{R}^4$ , then \eqref{3.2} is equivalent to
\be\label{3.3}
\Delta_5f(t,x)-f_t(t,x)=0.
\ee
Here $x=(x_1, x_2, x_3, x_4, z)$. Since the 5-dimensional heat kernel is
\be
G_5(t,x,y):=\frac{1}{(4\pi t)^{5/2}}e^{-\frac{|x-y|^2}{4t}},
\ee
the solution to equation \eqref{3.3} equipped with initial data $f(0,x)=f_0(x)$ has the following representation
\be
f(t,x)=\frac{1}{(4\pi t)^{5/2}}\int_{\mathbb{R}^5}e^{-\frac{|x-y|^2}{4t}}f_0(y)dy.
\ee
Coming back to axial-symmetric case, we assume that
\be
f_0(y)=f_0(\rho,l),
\ee
where $\rho=\sqrt{y_1^2+y_2^2+y_3^2+y_4^2}$, $l=y_5$, then we have
\be\label{1.6}
\begin{split}
f(t,r,z)=&\frac{1}{(4\pi t)^{5/2}}\int_{\mathbb{R}^5}e^{-\frac{|x'-y'|^2+(z-l)^2}{4t}}f_0(\rho,l)dy'dl\\
=&\frac{1}{(4\pi t)^{5/2}}\int_{\mathbb{R}^5}\exp\left(-\frac{r^2+\rho^2-2r\rho\cos\langle x', y'\rangle}{4t}\right)\cdot e^{-\frac{(z-l)^2}{4t}}f_0(\rho,l)dy'dl\\
=&\frac{1}{(4\pi t)^{5/2}}\int_{-\infty}^{+\infty}\int_0^{\infty}\exp\left(-\frac{r^2+\rho^2}{4t}\right)\cdot e^{-\frac{(z-l)^2}{4t}}f_0(\rho,l)\\
&\quad\quad\quad\quad\quad\cdot\left(\int_{|\omega|=1}\exp\left(\frac{r\rho\cos\langle x', y'\rangle}{2t}\right)d\omega\right)\rho^3d\rho dl.
\end{split}
\ee
Here $y'=\rho\cdot\omega$ and $\langle x',y'\rangle$ denotes the angle between vector $x'$ and $y'$. Now we define
\be\label{1.65}
I:=\int_{|\omega|=1}\exp\left(\frac{r\rho\cos\langle x', y'\rangle}{2t}\right)d\omega.
\ee
Since the integral above related only on the angle between $x'$ and $y'$, we assume $x'=(r,0,0,0)$ without loss of generality. In this case $\cos\langle x',y'\rangle=\frac{y_1}{\rho}$ and \eqref{1.65} equals
\be\label{1.7}
\begin{split}
I=&\int_{|\omega|=1}\exp\left(\frac{r\rho\cos\langle x',\omega\rangle}{2t}\right)d\omega\\
=&\int_{-1}^1\exp\left(\frac{r\rho s}{2t}\right)\cdot 4\pi (1-s^2)\frac{ds}{\sqrt{1-s^2}}\\
=&4\pi\int_{-1}^1\exp\left(\frac{r\rho s}{2t}\right)\sqrt{1-s^2}ds\\
=&4\pi^2\left(\frac{r\rho}{2t}\right)^{-1}\mathcal{I}_1\left(\frac{r\rho}{2t}\right).
\end{split}
\ee
Here and below, $\mathcal{I}_\alpha$ is the \emph{Modified Bessel function of first kind} with footnote $\alpha\in\mathbb{R}$. Last equality is due to the following lemma.
\begin{lemma}For any $A>0$, the following equation holds.
\be
J:=\int_{-1}^1e^{As}\sqrt{1-s^2}ds=\frac{1}{A}\mathcal{I}_1(A)
\ee
\end{lemma}

\begin{proof} Let $s=\sin\theta$, then
\[
\begin{split}
J=&\int_{-\pi/2}^{\pi/2}e^{A\sin\theta}\cos^2\theta d\theta\\
=&\int_{-\pi/2}^{\pi/2}e^{A\cos(\theta-\pi/2)}\cos^2\theta d\theta\\
=&\int_{-\pi}^{0}e^{A\cos\theta}\sin^2\theta d\theta\\
=&\int_{0}^{\pi}e^{A\cos\theta}\sin^2\theta d\theta.
\end{split}
\]
Using integration by parts,
\[
\begin{split}
J=&-\frac{1}{A}\int_0^\pi \sin\theta de^{A\cos\theta}\\
=&\frac{1}{A}\int_0^\pi e^{A\cos\theta}\cos\theta d\theta=\frac{1}{A}\mathcal{I}_1(A).
\end{split}
\]
The lase equality is due to equation (4) on pp. 181 \cite{Watson1995}.
\ef

Substituting \eqref{1.7} to \eqref{1.6}, we have
\be
f(t,r,z)=\int_{-\infty}^\infty\int_0^{+\infty}\frac{1}{4\sqrt{\pi}rt^{3/2}}\exp\left(-\frac{r^2+\rho^2+(z-l)^2}{4t}\right)\mathcal{I}_1\left(\frac{r\rho}{2t}\right)\rho^2f_0(\rho,l)d\rho dl.
\ee
Thus the heat kernel to equation \eqref{3.2} is
\be
\t{G}(t;r,\rho,z-l):= \frac{1}{4\sqrt{\pi}rt^{3/2}}\exp\left(-\frac{r^2+\rho^2+(z-l)^2}{4t}\right)\mathcal{I}_1\left(\frac{r\rho}{2t}\right)\rho.
\ee
Denoting $G$ the heat kernel to \eqref{3.1}, using the relation of $v$ and $f$, it follows
\be
G(t;r,\rho,z-l)=\frac{r}{\rho}\t{G}(t;r,\rho,z-l),
\ee
we have
\be
G(t;r,\rho,z-l)=\frac{1}{4\sqrt{\pi}t^{3/2}}\exp\left(-\frac{r^2+\rho^2+(z-l)^2}{4t}\right)\mathcal{I}_1\left(\frac{r\rho}{2t}\right).
\ee
That is, if \eqref{3.1} equipped with initial data $v(0,r,z)=v_0(r,z)$, we have
\be
v(t,r,z)=\int_{-\infty}^{+\infty}\int_0^{+\infty}G(t;r,\rho,z-l)\cdot v_0(\rho,l)\rho d\rho dl.
\ee
This proves \eqref{HeatKer}. By a direct integration on $t$, we can get \eqref{GreenFunc}.
\end{proof}

\subsection{Proof of the Weighted $L^p$ Estimates}
\subsubsection*{\noindent Calculation for \eqref{1.8} and \eqref{1.9} for $\boldsymbol{\dl\in(0,1)}$}

First, we come to calculate
\be
I:=\lt(\int^{+\i}_{-\i}\int^{+\i}_0|\G(r,\rho,z-l)|^p\f{1}{\rho} d\rho dl\rt)^{1/p}.\nn
\ee
From the formula \eqref{GreenFunc} and \eqref{HeatKer}, we can get
\be
\begin{aligned}
&I\ls \lt(\int^{+\i}_{-\i}\int^{+\i}_0\Big|\int^\i_0t^{-3/2}\exp (-\f{\rho^2+r^2+|z-l|^2}{4t})\mathcal{I}_1(\f{\rho r}{2t})dt\Big|^p\f{1}{\rho} d\rho dl\rt)^{1/p}\nn
\end{aligned}
\ee

Using the estimate \eqref{1.13}, we decompose the integration on $t$ into two parts: $t\geq \f{\rho r}{2}$ and $t<\f{\rho r}{2}$,
\be\label{3.16}
\begin{aligned}
&I\ls  \lt(\int^{+\i}_{-\i}\int^{+\i}_0\Big|\int^\i_{\f{\rho r}{2}}t^{-3/2}\exp\Big(-\f{\rho^2+r^2+|z-l|^2}{4t}\Big)\f{\rho r}{2t}dt\Big|^p\f{1}{\rho} d\rho dl\rt)^{1/p}\\
&\q+ \lt(\int^{+\i}_{-\i}\int^{+\i}_0\Big|\int^{\f{\rho r}{2}}_0t^{-3/2}\exp \Big(-\f{|\rho-r|^2+|z-l|^2}{4t}\Big)\s{\f{2 t}{\rho r}}dt\Big|^p\f{1}{\rho} d\rho dl\rt)^{1/p}\\
&\q\ls  \int^\i_0\lt(\int^{+\i}_{-\i}\int^{+\i}_0\Big|t^{-3/2}\exp \Big(-\f{\rho^2+r^2+|z-l|^2}{4t}\Big)\f{\rho r}{2t}\Big|^p\f{1}{\rho} d\rho dl\rt)^{1/p}dt\\
&\q+ \lt(\int^{+\i}_{-\i}\int^{+\i}_0\Big|(\rho r)^{-1/2}\int^{\f{2\rho r}{|\rho-r|^2+|z-l|^2}}_0s^{-1}\exp (-\f{1}{s})ds\Big|^p\f{1}{\rho} d\rho dl\rt)^{1/p},\\
\end{aligned}
\ee
where we have used the Minkovski inequality for the first term  and variable change for the second term on the righthand side of \eqref{3.16}. Continuing computations indicate that
\be
\begin{aligned}
&I\ls  r\int^\i_0t^{-5/2}\exp\big(-\f{r^2}{4t}\big)\lt(\int^{+\i}_{-\i}\int^{+\i}_0\exp \Big(-\f{p(\rho^2+|z-l|^2)}{4t}\Big)\rho^{p-1} d\rho dl\rt)^{1/p}dt\\
&\q+ \lt(\int^{+\i}_{-\i}\int^{+\i}_0(\rho r)^{-p/2}\Big(\f{2\rho r}{|\rho-r|^2+|z-l|^2}\Big)^{\ve p}\f{1}{\rho} d\rho dl\rt)^{1/p},\q (\ve>0),
\end{aligned}\nn
\ee
where we have used the fact $s^{-\ve}e^{-\f{1}{s}}\ls 1$ for $\ve>0$. By integration on $\rho$ and $l$, we obtain
\be
\begin{aligned}
&I\ls  r\int^\i_0t^{-5/2}\exp\big(-\f{r^2}{4t}\big)t^{\f{p+1}{2p}}dt\\
&\q\q+ r^{-1/2+\ve}\lt(\int^{+\i}_{-\i}\int^{+\i}_0\f{1}{(|\rho-r|+|z-l|)^{2\ve p}}\f{1}{\rho^{p/2-\ve p+1}} d\rho dl\rt)^{1/p}\\
&\q\ls  r^{\f{1}{p}-1}\int^\i_0e^{-s}s^{-1/2p}ds\\
&\q\q+ r^{-1/2+\ve}\lt(\int^{+\i}_0\f{1}{|\rho-r|^{2\ve p-1}}\f{1}{\rho^{p/2-\ve p+1}} d\rho \rt)^{1/p}\q(\text{assume}\ 2\ve p-1>0)\\
&\q\ls  r^{\f{1}{p}-1},
\end{aligned}
\ee
provided that
\be
\lt\{
\begin{aligned}
&1/2p<1,\\
&0<2\ve p-1<1,\\
&p/2-\ve p+1<1,\\
&p/2-\ve p+1+2\ve p-1>1,
\end{aligned}
\rt.\nn
\ee
which is equivalent to
\be\label{3.18}
\lt\{
\begin{aligned}
&p>1/2,\\
&\f{1}{1/2+\ve}<p<\f{1}{\ve},\\
&\ve>\f{1}{2},
\end{aligned}
\rt.
\ee
For any $1\leq p<2$, we can choose an $\ve$ such that \eqref{3.18} is satisfied.\ef

From \eqref{HeatKer}, we have
\be
\p_zG(t;r,\rho,z-l)=C\frac{z-l}{t^{5/2}}\cdot\exp\left(-\frac{r^2+\rho^2+(z-l)^2}{4t}\right)\mathcal{I}_1\left(\frac{r\rho}{2t}\right).\nn
\ee
Define $J:=\int^{+\i}_{-\i}\int^{+\i}_0|\p_z\G(r,\rho,z-l)|\f{1}{\rho^\dl}d\rho dl$. The estimate of $J$ will be essentially the same as $I$ by  one more $t^{-1/2}$ coming out. We give a brief review. First we see that
\be
|\p_zG(t;r,\rho,z-l)|\ls t^{-2}\cdot\exp\left(-\frac{r^2+\rho^2}{4t}-\f{(z-l)^2}{8t}\right)\mathcal{I}_1\left(\frac{r\rho}{2t}\right).\\
\ee
Then,
\be
J\ls \int^{+\i}_{-\i}\int^{+\i}_0\Big|\int^\i_0t^{-2}\exp\left(-\frac{r^2+\rho^2}{4t}-\f{(z-l)^2}{8t}\right)\mathcal{I}_1\left(\frac{r\rho}{2t}\right)dt\Big|\f{1}{\rho^\dl} d\rho dl.\nn
\ee
Almost the same as \eqref{3.16}, we have
\be
\begin{aligned}
J&\ls  \int^{+\i}_{-\i}\int^{+\i}_0\left|\int^\i_{\f{\rho r}{2}}t^{-2}\exp\left(-\frac{r^2+\rho^2}{4t}-\f{(z-l)^2}{8t}\right)\f{\rho r}{2t}dt\right|\f{1}{\rho^\dl} d\rho dl\\
&\q+ \int^{+\i}_{-\i}\int^{+\i}_0\left|\int^{\f{\rho r}{2}}_0t^{-2}\exp \Big(-\f{|\rho-r|^2+1/2|z-l|^2}{4t}\Big)\s{\f{2 t}{\rho r}}dt\right|\f{1}{\rho^\dl} d\rho dl\\
&\ls  \int^\i_0\int^{+\i}_{-\i}\int^{+\i}_0\left|t^{-2}\exp \Big(-\f{\rho^2+r^2+1/2|z-l|^2}{4t}\Big)\f{\rho r}{2t}\right|\f{1}{\rho^\dl} d\rho dldt\\
&\q+ \int^{+\i}_{-\i}\int^{+\i}_0\left|\f{1}{\s{\rho r(|\rho-r|^2+1/2|z-l|^2)}}\right.\\
&\hskip 5cm\cdot\left.\int^{\f{2\rho r}{|\rho-r|^2+1/2|z-l|^2}}_0s^{-3/2}\exp (-\f{1}{s})ds\right|\f{1}{\rho^\dl} d\rho dl.\\
\end{aligned}
\ee
Continuing computations indicate that, for $\ve>0$,
\be
\begin{aligned}
J&\ls  r\int^\i_0t^{-3}\exp\big(-\f{r^2}{4t}\big)\int^{+\i}_{-\i}\int^{+\i}_0\exp \Big(-\f{\rho^2+1/2|z-l|^2}{4t}\Big)\rho^{1-\dl} d\rho dldt\\
&\q+ \int^{+\i}_{-\i}\int^{+\i}_0\f{1}{\s{\rho r(|\rho-r|^2+1/2|z-l|^2)}}\Big(\f{2\rho r}{|\rho-r|^2+1/2|z-l|^2}\Big)^{\ve}\f{1}{\rho^\dl} d\rho dl\\
&\ls  r\int^\i_0t^{-3}\exp\big(-\f{r^2}{4t}\big)t^{\f{3-\dl}{2}}dt\\
&\q\q+ r^{-1/2+\ve}\int^{+\i}_{-\i}\int^{+\i}_0\f{1}{(|\rho-r|+|z-l|)^{1+2\ve}}\f{1}{\rho^{1/2-\ve +\dl}} d\rho dl\\
&\ls  r^{-\dl}\int^\i_0e^{-s}s^{\f{\dl-1}{2}}ds\\
&\q+ r^{-1/2+\ve}\int^{+\i}_0\f{1}{|\rho-r|^{2\ve }}\f{1}{\rho^{1/2+\dl-\ve}} d\rho \\
&\ls  r^{-\dl},
\end{aligned}\nn
\ee
provided that
\be
\lt\{
\begin{aligned}
&2\ve<1,\\
&0<1/2+\dl-\ve<1,\\
&2\ve+1/2+\dl-\ve>1,
\end{aligned}
\rt.\nn
\ee
which is equal to
\be\label{3.18-1}
|\dl-1/2|<\ve<\frac{1}{2}.
\ee
When $\dl\in(0,1)$, we can choose an $\ve$ such that \eqref{3.18-1} is satisfied. \ef

\vskip 0.3 cm
\subsubsection*{ Calculation for \eqref{1.10} and \eqref{1.9} for $\boldsymbol{\dl=0}$}

Recall
\be
\G(r,\rho,z-l)=\int_0^\infty\frac{1}{4\sqrt{\pi}t^{3/2}}\exp\left(-\frac{r^2+\rho^2+(z-l)^2}{4t}\right)\mathcal{I}_1\left(\frac{r\rho}{2t}\right)dt,\nn
\ee
and by using Minkovski inequality, it follows that
\be
\begin{split}
&\q\left(\int_{-\infty}^\infty\int_0^\infty|\G(r,\rho,z-l)|^2\rho d\rho dl\right)^{1/2}\\
&\leq\int_0^\infty\left(\int_{-\infty}^\infty\int_0^\infty|G(t;r,\rho,z-l)|^2\rho d\rho dl\right)^{1/2}dt\\
&=\int_0^\infty\left(\int_{-\infty}^\infty\int_0^\infty|G(t;r,\rho,z-l)|^2 dl\rho d\rho \right)^{1/2}dt\\
&\lesssim\int_0^\infty\left(\int_{-\infty}^\infty\int_0^\infty\exp\left(-\frac{r^2+\rho^2}{2t}\right)\mathcal{I}_1^2\left(\frac{r\rho}{2t}\right)t^{-5/2} \rho d\rho \right)^{1/2}dt\\
&=\int_0^\infty t^{-3/4}\exp\left(-\frac{r^2}{8t}\right)\sqrt{\mathcal{I}_1\left(\frac{r^2}{4t}\right)}dt\\
&=\sqrt{r}\cdot\int_0^\infty s^{-3/4}\exp\left(-\frac{1}{8s}\right)\sqrt{\mathcal{I}_1\left(\frac{1}{4s}\right)}ds\\
&\lesssim\sqrt{r}.
\end{split}
\ee
Here, we have applied identity \eqref{ID2} to prove the 5th line. The last line holds because the asymptotic behavior of function $\mathcal{I}_1$ (see \eqref{1.13}) makes the the integral in this line convergent. This finish the proof of \eqref{1.10}

Next, integrating $\p_zG(t;r,\rho,z-l)$ with respect to $l$, it follows that
\be
\int_{0}^\infty|\p_zG(t;r,\rho,z-l)|dl\leq \frac{C}{t^{3/2}}\exp\left(-\frac{r^2+\rho^2}{4t}\right)\mathcal{I}_1\left(\frac{r\rho}{2t}\right).\nn
\ee
Therefore, using identity \eqref{ID1} in \textbf{Appendix}, integrating with respect to $\rho$ to get
\be
\int_0^\infty\int_{-\infty}^\infty|\p_zG(t;r,\rho,z-l)|dl d\rho\leq\frac{C}{r\sqrt{t}}\left(1-e^{-\frac{r^2}{4t}}\right).\nn
\ee
 When $\dl=0$, \eqref{1.9} is proved by integrating with $t$ on $(0,\infty)$.

\subsection{Proof of Corollary \ref{Cor1} and Corollary \ref{Cor2}}\par

\subsubsection*{Proof of Corollary \ref{Cor1}}

Using the divergence free condition of axially symmetric Navier-Stokes equation, it follows that
\be
\p_r(ru^r)+\p_z(ru^z)=0.
\ee
Then there exists a scaler function $L^\theta=L^\theta(t,r,z)$ s.t.
\be\label{psi3.2}
-\p_zL^\theta=u^r,\quad\text{and}\quad\frac{1}{r}\p_r(rL^\theta)=u^z.
\ee
Therefore, after substituting \eqref{psi3.2} in identity $\omega^\theta=\p_zu^r-\p_ru^z$, we have
\be
w^\theta=\p_zu^r-\p_ru^z=-\p^2_zL^\theta-\p_r\left(\frac{1}{r}\p_r(rL^\theta)\right)=-\left(\Delta-\frac{1}{r^2}\right)L^\theta.
\ee
That is, $L^\theta$ satisfies
\be\label{EQN3.4}
-\left(\Delta-\frac{1}{r^2}\right)L^\theta=w^\theta.
\ee
By using the Green function in\eqref{1.6} elliptic equation \eqref{EQN3.4} is solved by
\be\label{def3.6}
L^\theta(t,r,z):=\int_{-\infty}^\infty\int_0^\infty \G(r,\rho,z-l)w^\theta(t,\rho,l)\rho d\rho dl.
\ee

We claim that $L^\theta$ is uniformly bounded, the proof is: applying assumption
\be
|w^\th|\ls \f{1}{r^2},\nn
\ee
 one has
\be
\begin{split}
|L^\theta(t,r,z)|&\leq\int_{-\infty}^\infty\int_0^\infty\Gamma(r,\rho,z-l)\cdot|w^\theta(t,\rho,l)|\rho d\rho dl\\
&\lesssim\int_0^\infty\int_{-\infty}^\infty\Gamma(r,\rho,z-l)\rho^{-1}d\rho dl.\\
&\ls 1.
\end{split}
\ee
which is a direct consequence of \eqref{1.8} when we set $p=1$. Therefore, $L^\theta\in L^\infty_t(BMO(\mathbb{R}^3))$. This means
\be
b=u^re_r+u^ze_z=\nabla\times(L^\theta\cdot\mathbf{e_\theta})\in L_t^\infty(BMO^{-1}(\mathbb{R}^3)).
\ee
The conclusion in \cite{LZ2} implies the regularity of $u$. \ef

\subsubsection*{Proof of Corollary \ref{Cor2}}

Combining the first identity of \eqref{psi3.2} and \eqref{def3.6}, $u^r$ has the following representation formula
\be
u^r=-\int_{-\infty}^\infty\int_0^\infty\p_z\G(r,\rho,z-l)w^\theta(t,\rho,l)\rho d\rho dl
\ee
Therefore, as a direct consequence of \eqref{1.9} when we assume $|w^\th|\ls \f{1}{r^{1+\dl}}$, we have,
\be
 |u^r|\leq \int_{-\infty}^\infty\int_0^\infty|\p_z\G(r,\rho,z-l)|\f{1}{\rho^\dl} d\rho dl\ls \f{1}{r^\dl}.\nn
\ee
The conclusion in \cite{Px2} implies the regularity of $u$. \ef
\section{Appendix}
\subsection*{A brief introduction of modified Bessel function $\mathcal{I}_\alpha(x)$}
Modified Bessel functions $\mathcal{I}_\alpha$ (first kind) and $\mathcal{K}_\alpha$ (second kind) are two linearly independent solutions of the modified Bessel equation
\be
x^{2}{\frac {d^{2}y}{dx^{2}}}+x{\frac {dy}{dx}}-(x^{2}+\alpha ^{2})y=0.
\ee
In this article, we only consider $\mathcal{I}_1$, which satisfies the following asymptotic behavior

\be
\mathcal{I}_1\left(x\right)\lesssim
\left\{
\begin{array}{ll}
x&,\quad 0<x\leq 1;\\
\frac{e^x}{\sqrt{x}}&,\quad x>1.\\
\end{array}
\right.
\ee
We refer readers to \cite{AS1} for more details. The following lemma contains two identities related to $\mathcal{I}_1$, which was applied in Section \ref{SEC3}.
\begin{lemma}
For any $a>0$, the following identities hold
\be\label{ID1}
\int_0^\infty e^{-s^2}\mathcal{I}_1(as)ds=\frac{1}{a}\left(e^{\frac{a^2}{4}}-1\right);
\ee
\be\label{ID2}
\int_0^\infty e^{-s^2}\cdot\mathcal{I}_1^2(as)\cdot sds=\frac{1}{2} e^{\frac{a^2}{2}}\cdot\mathcal{I}_1\left(\frac{a^2}{2}\right).
\ee
\end{lemma}
\begin{proof}
We only prove identity \eqref{ID1}, while the method of the proof of \eqref{ID2} is essentially the same. Consider the definition of modified Bessel function (see \cite{AS1} p. 375 for details)
\be
\mathcal{I}_1(s)=\sum_{m=0}^\infty\frac{1}{m!(m+1)!}\left(\frac{s}{2}\right)^{2m+1}.
\ee
It follows that
\be
\begin{split}
\int_0^\infty e^{-s^2}\mathcal{I}_1(as)ds=&\sum_{m=0}^\infty\frac{1}{m!(m+1)!}\int_0^\infty\left(\frac{as}{2}\right)^{2m+1}e^{-s^2}ds\\
=&\frac{1}{2}\sum_{m=0}^\infty\frac{\left(\frac{a}{2}\right)^{2m+1}}{(m+1)(m!)^2}\int_0^\infty e^{-s}s^mds\\
=&\frac{1}{a}\sum_{m=0}^\infty\frac{1}{(m+1)!}\left(\frac{a}{2}\right)^{2m+2}\\
=&\frac{1}{a}\sum_{m=0}^\infty\frac{1}{(m+1)!}\left(\frac{a^2}{4}\right)^{m+1}\\
=&\frac{1}{a}\left(e^{\frac{a^2}{4}}-1\right).
\end{split}
\ee
\end{proof}
\indent

\section*{Acknowledgments}

The authors wish to thank Prof. Huicheng Yin in Nanjing Normal University and Prof. Qi S. Zhang in UC Riverside for their constant encouragement on this topic. They are also grateful to department of mathematics, UC Riverside for its hospitality during their visits.

The first author is supported by China Scholar Council (NO. 201606190089). The second author is supported by the Research Foundation of Nanjing University of Aeronautics and Astronautics (NO. 1008-YAH17070).


\end{document}